\documentclass[dvipsnames]{article}
\usepackage[margin=1in]{geometry}
\usepackage[utf8]{inputenc}
\usepackage[english]{babel}
\usepackage{graphicx,hyperref,amsmath,amssymb, multicol, subcaption, tikz, pgf, amsthm, xr, mathrsfs, xcolor, enumerate}
\usepackage{pdfpages}
\usetikzlibrary{arrows, calc}
\usepackage{authblk}

\usepackage{lineno} 
\usepackage{subfiles}
\usepackage{xcolor}
\usepackage{blindtext}
\usepackage{ulem}

\newtheorem{thm}{Theorem}[section]
\newtheorem{prop}[thm]{Proposition}
\newtheorem{cor}[thm]{Corollary}
\newtheorem{lem}[thm]{Lemma}

\newtheorem{obs}[thm]{Observation}

\theoremstyle{definition}
\newtheorem{defn}[thm]{Definition}
\newtheorem{ex}[thm]{Example}

\newtheorem{innercustomthm}{Theorem}
\newenvironment{customthm}[1]
  {\renewcommand\theinnercustomthm{#1}\innercustomthm}
  {\endinnercustomthm}

\newtheorem{innercustomcor}{Corollary}
\newenvironment{customcor}[1]
  {\renewcommand\theinnercustomcor{#1}\innercustomcor}
  {\endinnercustomcor}

 \newcommand{\bpf}{\begin{proof}}
 \newcommand{\epf}{\end{proof}}

 \newcommand{\aw}{\textup{aw}}
 \newcommand{\dist}{\textup{d}}
 \newcommand{\diam}{\textup{diam}}
 \newcommand{\rad}{\textup{rad}}

\title{Full classification of anti-van der Waerden numbers of graph products of forests}
\author[1]{Zhanar Berikkyzy}
\author[2]{Joe Miller}
\author[3]{Nathan Warnberg}

\affil[1]{Mathematics Department, Fairfield University\{zberikkyzy@fairfield.edu\}}
\affil[2]{Department of Mathematics, Iowa State University \{jmiller0@iastate.edu\}}
\affil[3]{Department of Mathematics and Statistics, University of Wisconsin-La Crosse, \{nwarnberg@uwlax.edu\}}

\date{}
 
\begin{document}
\maketitle

\begin{abstract}
    The anti-van der Waerden number of a graph $G$ is the fewest number of colors needed to guarantee a rainbow $3$-term arithmetic progression in $G$, denoted $\aw(G,3)$. It is known that the anti-van der Waerden number of graph products is $3 \le \aw(G\square H,3)\le 4$.  Previous work has been done on classifying families of graph products into $\aw(G\square H,3) = 3$ and $\aw(G\square H,3) = 4$.  Some of these families include the product of two paths, the product of paths and cycles, the product of two cycles, and the product of odd cycles with any graph.  Recently, a partial characterization of the product of two trees was established.  This paper completes the characterization for $\aw(T\square T',3)$ where $T$ and $T'$ are trees.  Moreover, this result extends to a full classification of products of forests.
    
\end{abstract}

\begin{footnotesize}


\end{footnotesize}

\section{Introduction}

Ramsey Theory is a branch of mathematics that assigns colors to elements of a set and then determines if monochromatic substructures exist within the set, whereas Anti-Ramsey Theory determines if rainbow (polychromatic) substructures exist within the set.  Ramsey Theory has a long history dating back to around $1920$ with Schur, Ramsey and van der Waerden making the earliest contributions (see \cite{R,S,W27}) where the sets considered were integers or graph edges.  It was not until $1973$ when when Erd{\H{o}}s, Simonovits, and S{\'{o}}s, in \cite{ESS}, introduced the idea of Anti-Ramsey Theory.  Thirty years later, Jungi\'c et al started investigating anti-van der Waerden problems where the sets being colored were $\{1,2,\dots,n\} = [n]$ and $\mathbb{Z}_n$ and the rainbow substructures were $3$-term arithmetic progressions (see \cite{Jungicsecond, Jungicfirst}).  These papers focused on the conditions on the sizes of the color classes that guarantee rainbow arithmetic progressions.  The anti-van der Waerden number was first defined in \cite{U} by Uherka in 2013.  
The question asked was: given a fixed value of $n$, what is the fewest number of colors needed to guarantee a rainbow $k$-term arithmetic progression in $[n]$ or $\mathbb{Z}_n$, denoted $\aw([n],k)$ and $\aw(\mathbb{Z}_n,k)$, respectively?
In \cite{DMS}, Butler et al. bounded $\aw(\mathbb{Z}_n,3)$ based on the prime factorization of $n$ and found a logarithmic bound for $\aw([n],3)$.  The exact values of $\aw([n],3)$ were then determined by Berikkyzy, Schulte and Young in \cite{BSY}. 
Young, in \cite{finabgroup}, determined the anti-van der Waerden numbers for finite abelian groups based on the order of the group.  Concurrently, authors noted that a $3$-term arithmetic progression $(a_1,a_1+d, a_1+2d)$ satisfies the equation $x_1 + x_3 = 2x_2$ and anti-van der Waerden numbers for Sidon sets and other linear equations were investigated (see \cite{ATHNWY}, \cite{BK}, \cite{BKKTTY}).

Investigations then turned to arithmetic progressions in graphs as it was observed that the path $P_n$ behaves like $[n]$ and the cycle $C_n$ behave like $\mathbb{Z}_n$.

 Given a graph $G$, an {\it exact $r$-coloring of $G$} is a surjective function $c:V(G) \to [1,\dots,r]$.  An {\it arithmetic progression} in $G$ of length $k$ ($k$-AP) with common difference $d$ is a list of vertices $(v_1,\dots, v_k)$ such that $\dist(v_i,v_{i+1}) = d$ for $1\le i < k$.  
 An arithmetic progression is {\it rainbow} if all of the vertices are colored distinctly.  
 The fewest number of colors that guarantees a rainbow $k$-AP is called the {\it anti-van der Waerden number of $G$} and is denoted $\aw(G,k)$. To show $r\le \aw(G,k)$ we construct an exact $(r-1)$-coloring that avoids rainbow $k$-APs.  To show $\aw(G,k) \le r$, we show that every exact $r$-coloring gives a rainbow $k$-AP.

 The definition of $\aw(G,k)$ was introduced in \cite{RSW} and provided the first results about graph products.  In particular, Theorems \ref{PmxPn} and \ref{theorem:rsw} have proven to be essential in most results about $\aw(G\square H,k)$.

\begin{thm}[\cite{RSW}]\label{PmxPn}
    For $m,n \geq 2$, \[\aw(P_m \square P_n, 3) = \begin{cases}
3 & \text{if $m = 2$ and $n$ is even, or $m = 3$ and $n$ is odd,} \\
4 & \text{otherwise.}
\end{cases}\]
\end{thm}
A subgraph $H$ of $G$ is {\it isometric} if for every $u,v\in V(H)$ we have $d_H(u,v) = d_G(u,v)$.
Theorem \ref{PmxPn} is used when we find an isometric subgraph $P_m\square P_n$ within $G\square H$ to either show that our coloring must have a rainbow within the isometric subgraph or that we know we can color the isometric subgraph in some way and avoid rainbows.  This result was used extensively when determining $\aw(P_m\square C_n,3)$ in \cite{MW}.

\begin{thm}[\cite{RSW}]\label{theorem:rsw}
    If $G$ and $H$ are connected graphs and $|G|,|H| \ge 2$, then $\aw(G\square H,3) \le 4$.
\end{thm}

Authors investigated $\aw(G,k)$ on trees and graphs with small diameter in \cite{SWY}, graph products of paths and cycles in \cite{MW}, and graph products of trees $T$ and $T'$ where $\diam(T\square T)$ is odd in \cite{BMSWW}.  

This paper completes the classification of $\aw(T\square T',3)$ by considering when $\diam(T\square T')$ is even. The following is the main result of this paper.

\begin{thm}\label{thm:treeprodeven}
    Let $T$ and $T'$ be nontrivial trees, where $\diam(T\square T')$ is even. Then 
    \[\aw(T\square T',3) \hspace{-.05cm}=\hspace{-.05cm}
\begin{cases}
    3 & \text{if $T$ or $T'$ is weakly non-$3$-peripheral or isomorphic to $P_2$}, \\
    4 & \text{if $T$ and $T'$ are both strongly non-$3$-peripheral and not $P_2$}.
\end{cases}\]
\end{thm}

We summarize the full classification of anti-van der Waerden numbers of graph products of trees in the following corollary.

\begin{cor}\label{cor:full}
Let $T$ and $T'$ be trees.  Then,
\[\aw(T\square T',3) \hspace{-.05cm}=\hspace{-.1cm} 
\begin{cases}
     & \text{if $T$ or $T'$ is $3$-peripheral, or}\\

    3& \text{$\diam(T\square T')$ is even and $T$ or $T'$ is $P_2$, or}\\

    & \text{$\diam(T\square T')$ is even and $T$ or $T'$ is weakly non-$3$-peripheral,}\\
    4 & \text{otherwise}.
\end{cases}\]

\end{cor}

The paper is organized as follows.  In Section \ref{sec:2}, definitions, notation and conventions are established along with several results that support the main theorem of the paper.  Section \ref{sec:3} provides a case-analysis for $\aw(T\square T',3)$ based on properties of $T$ and $T'$, e.g. whether the trees are strongly or weakly $3$-peripheral (which we define in Section \ref{sec:2}), and concludes with the main result.

\section{Preliminary Results}\label{sec:2}

This section introduces the tools needed to prove the main result.  
We begin with basic definitions and known results.

If $G = (V,E)$ and $H = (V', E')$ then the {\it Cartesian
product}, written $G\square H$, has vertex set $\{(x, y) : x \in V \text{ and } y \in V' \}$ and $(x, y)$ and
$(x', y')$ are adjacent in $G \square H$ if either $x = x'$ and $yy' \in E'$ or $y = y'$ and $xx' \in E$. This paper will use the convention that if \[V(G) = \{u_1,\dots, u_{n_1}\} \quad \text{and} \quad V(H) = \{w_1,\dots,w_{n_2}\},\] then $V(G\square H) = \{v_{1,1},\dots, v_{n_1,n_2}\}$ where $v_{i,j}$ corresponds to the vertices $u_i \in V(G)$ and $w_j \in V(H)$.  Also, if $1\leq i \leq n_2$, then $G_i$ denotes the $i$th labeled copy of $G$ in $G \square H$. Likewise, if $1 \leq j \leq n_1$, then $H_j$ denotes the $j$th labeled copy of $H$ in $G \square H$.  In other words, $G_i$ is the induced subgraph $G_i = G\square H[\{v_{1,i},\dots, v_{n_2,i}\}]$, and $H_j$ is the induced subgraph $H_j = G\square H[\{v_{j,1}, \dots, v_{j,n_1}\}]$.  Notice that the $i$ subscript in $G_i$ corresponds to the $i$th vertex of $H$ and the $j$ in the subscript in $H_j$ corresponds to the $j$th vertex of $G$. See Figure \ref{fig:cex} below for an example where $G = P_4$ and $H$ is a broom graph.

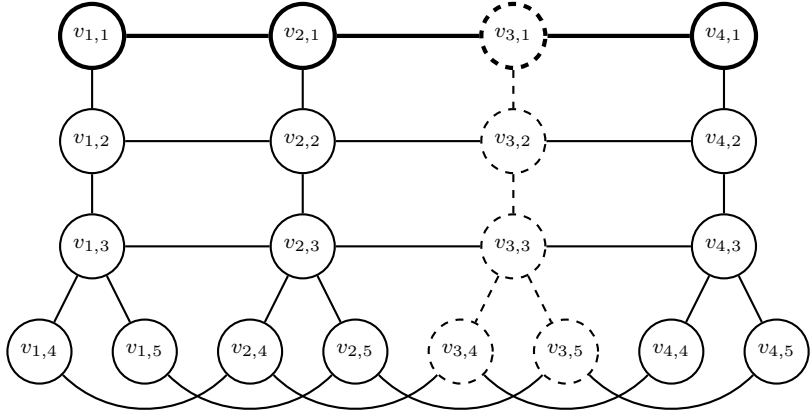
\begin{figure}
    \centering
    \begin{tikzpicture}[scale = .7]
        \node[draw,circle,ultra thick] (11) at (0,6) {\footnotesize $v_{1,1}$};
        \node[draw,circle,thick] (12) at (0,4) {\footnotesize $v_{1,2}$};
        \node[draw,circle,thick] (13) at (0,2) {\footnotesize $v_{1,3}$};
        \node[draw,circle,thick] (14) at (-1,0) {\footnotesize $v_{1,4}$};
        \node[draw,circle,thick] (15) at (1,0) {\footnotesize $v_{1,5}$};
        
        \node[draw,circle,ultra thick] (21) at (4,6) {\footnotesize $v_{2,1}$};
        \node[draw,circle,thick] (22) at (4,4) {\footnotesize $v_{2,2}$};
        \node[draw,circle,thick] (23) at (4,2) {\footnotesize $v_{2,3}$};
        \node[draw,circle,thick] (24) at (3,0) {\footnotesize $v_{2,4}$};
        \node[draw,circle,thick] (25) at (5,0) {\footnotesize $v_{2,5}$};
        
        \node[draw,circle,ultra thick, dashed] (31) at (8,6) {\footnotesize $v_{3,1}$};
        \node[draw,circle,thick, dashed] (32) at (8,4) {\footnotesize $v_{3,2}$};
        \node[draw,circle,thick, dashed] (33) at (8,2) {\footnotesize $v_{3,3}$};
        \node[draw,thick, circle, dashed] (34) at (7,0) {\footnotesize $v_{3,4}$};
        \node[draw,thick, circle, dashed] (35) at (9,0) {\footnotesize $v_{3,5}$};
        
        \node[draw,circle,ultra thick] (41) at (12,6) {\footnotesize $v_{4,1}$};
        \node[draw,circle,thick] (42) at (12,4) {\footnotesize $v_{4,2}$};
        \node[draw,thick, circle] (43) at (12,2) {\footnotesize $v_{4,3}$};
        \node[draw,circle,thick] (44) at (11,0) {\footnotesize $v_{4,4}$};
        \node[draw,circle,thick] (45) at (13,0) {\footnotesize $v_{4,5}$};
        
        \draw[thick]  (11) to node [auto] {} (12);
        \draw[thick]  (12) to node [auto] {} (13);
        \draw[thick]  (13) to node [auto] {} (14);
        \draw[thick]  (13) to node [auto] {} (15);
        
        \draw[thick]  (21) to node [auto] {} (22);
        \draw[thick]  (22) to node [auto] {} (23);
        \draw[thick]  (23) to node [auto] {} (24);
        \draw[thick]  (23) to node [auto] {} (25);
        
        \draw[thick, dashed]  (31) to node [auto] {} (32);
        \draw[thick, dashed]  (32) to node [auto] {} (33);
        \draw[thick, dashed]  (33) to node [auto] {} (34);
        \draw[thick, dashed]  (33) to node [auto] {} (35);
        
        \draw[thick]  (41) to node [auto] {} (42);
        \draw[thick]  (42) to node [auto] {} (43);
        \draw[thick]  (43) to node [auto] {} (44);
        \draw[thick]  (43) to node [auto] {} (45);
        
        \draw[ultra thick]  (11) to node [auto] {} (21);
        \draw[thick]  (12) to node [auto] {} (22);
        \draw[thick]  (13) to node [auto] {} (23);
        \draw[thick, bend right = 45]  (14) to node [auto] {} (24);
        \draw[thick, bend right = 45]  (15) to node [auto] {} (25);
        
        \draw[ultra thick]  (31) to node [auto] {} (21);
        \draw[thick]  (32) to node [auto] {} (22);
        \draw[thick]  (33) to node [auto] {} (23);
        \draw[thick, bend left = 45]  (34) to node [auto] {} (24);
        \draw[thick, bend left = 45]  (35) to node [auto] {} (25);
        
        \draw[ultra thick]  (31) to node [auto] {} (41);
        \draw[thick]  (32) to node [auto] {} (42);
        \draw[thick]  (33) to node [auto] {} (43);
        \draw[thick, bend right = 45]  (34) to node [auto] {} (44);
        \draw[thick, bend right = 45]  (35) to node [auto] {} (45);
    \end{tikzpicture}

    \caption{The product of $G = P_4$ and broom graph $H$.  The subgraph $G_1$ is bolded and the subgraph $H_3$ is dashed. Graph $H$ is an example of a weakly non-$3$-peripheral tree of odd diameter (see Definition \ref{defn:weak/strong non-3-per}) while the path $G$ is an example of a strongly non-$3$-peripheral graph of odd diameter.}
    
    \label{fig:cex}
\end{figure}

\begin{prop}[\cite{MW}]\label{proposition:dist}
If $v_{i,j},v_{h,k} \in V(G \square H)$, then \[\dist_{G\square H}(v_{i,j},v_{h,k}) = \dist_G(u_i,u_h) + \dist_H(w_j,w_k).\]
\end{prop}

Note that if $G$ and $H$ are graphs, then a direct consequence of Proposition \ref{proposition:dist} is that \[\diam(G\square H) = \diam(G) + \diam(H).\]

\begin{cor}[\cite{MW}]\label{cor:isosubprod}
    If $G'$ is an isometric subgraph of $G$ and $H'$ is an isometric subgraph of $H$, then $G'\square H'$  is an isometric subgraph of $G\square H$. 
\end{cor}

Corollary \ref{cor:isosubprod} is particularly useful, as it is often used to find a rainbow structure in the graph product by only considering a small subgraph that preserves distances. 
Lemma \ref{isometricpathorC3} takes this a step further when we have exact $3$-colorings.  
Since we are applying the lemma to products of trees, we know that $C_3$ subgraphs do not exist so an isometric path exists that has all three colors.  
In practice, we take a shortest such path which implies that endpoints are uniquely colored and all interior vertices colored the same. This allows us to find structure within an arbitrary coloring. 

\begin{lem}[\cite{RSW}]\label{isometricpathorC3}
    If $G$ is a connected graph on at least three vertices with an exact $r$-coloring $c$ where $r \ge 3$, then there exists a subgraph $G'$ in $G$ with at least three colors where $G'$ is either an isometric path or $G' = C_3$.
\end{lem}

\begin{lem}[\cite{MW}]\label{|c(Hi)|<3}
    If $G$ and $H$ are connected, $|G|,|H| \ge 2$ and $c$ is an exact $r$-coloring of $G\square H$, $3\le r$, that avoids rainbow $3$-APs, then $|c(V(G_i))| \leq 2$ for $1 \leq i \leq |H|$.
\end{lem}

The following results are used to derive structural properties of rainbow-free colorings.

\begin{cor}[\cite{MW}]\label{|c(V(Gi U Gj))|<3}
    If $G$ and $H$ are connected graphs, $|G| \ge 2$, $|H|\ge 3$, $c$ is an exact, rainbow-free $r$-coloring of $G\square H$ with $r\ge 3$, and $v_iv_j \in E(H)$, then \[|c(V(G_i)\cup V(G_j))| \leq 2.\]
\end{cor}

\begin{prop}[\cite{MW}]\label{prop:everycopy,|H|>2}
    If $G$ and $H$ are connected graphs, $|G| \ge 2$, $|H|\ge 3$, $c$ is an exact, rainbow-free $r$-coloring of $G\square H$ with $r\ge 3$, then there is a color in $c(G\square H)$ that appears in every copy of $G$.
\end{prop}

The following proposition guarantees existence of a dominating color in a rainbow-free coloring. This observation is helpful in constructing such coloring, when it exists.

\begin{prop}\label{prop:everycopy}
    If $G$ and $H$ are connected graphs, $|G| \ge 2$, $|H|\ge 2$, $c$ is an exact, rainbow-free $r$-coloring of $G\square H$ with $r\ge 3$, then there is a color in $c(G\square H)$ that appears in every copy of $G$.
\end{prop}

\begin{proof}
    The case of $|H|\geq 3$ is handled by \ref{prop:everycopy,|H|>2}.
    So, suppose $|H|=2$, say $V(H) = \{w_1,w_2\}$. Assume $G_1$ and $G_2$ do not share a color. The pigeonhole principle implies that some copy of $G$ must have at least two colors, say $G_1$ has at least two colors.  Then there must exist adjacent vertices $v_{1,k},v_{1,\ell} \in V(G_1)$ such that $c(v_{1,k})\neq c(v_{1,\ell})$. Since $G_1$ and $G_2$ do not share a color, the $3$-AP $(v_{1,k},v_{1,\ell},v_{2,\ell})$ is rainbow, a contradiction. Thus, $G_1$ and $G_2$ share a color, as desired.
\end{proof}

Many conditions in this paper are about peripheral vertices, which we define next. For common graph theory terminology, see \cite{West}.
For a vertex $v$ in a connected graph $G$ the {\it eccentricity} of $v$, denoted $\epsilon(v)$, is the distance between $v$ and a vertex furthest from $v$ in $G$. 
If a vertex has minimum eccentricity, we call it a {\it central vertex}.  
The {\it radius} of $G$, denoted $\operatorname{rad}(G)$, is the eccentricity of any central vertex. 
The collection of all central vertices in $G$ is the {\it center} of $G$. 
If a vertex $v$ has $\epsilon(v) = \diam(G)$ we call $v$ a {\it peripheral} vertex. 
If a graph $G$ contains vertices $u_1,\ldots, u_n$ such that $\dist(u_i,u_j) = \diam(G)$ for all distinct $i,j \in \{1,\ldots,n\}$, then we call $G$ {\it $n$-peripheral}. 
A graph is {\it non-$n$-peripheral} if we cannot find $n$ vertices that are pairwise diameter away from each other.
Specifically, we focus on graphs that are $3$-peripheral and graphs that are non-$3$-peripheral.

\begin{thm}[\cite{BMSWW}]\label{TxG 3-per}
    If $T$ is a $3$-peripheral tree and $G$ is connected with $2\le |G|$, then \[\aw(T \square G, 3) = 3.\]
\end{thm}

\begin{thm}[\cite{BMSWW}]\label{theorem:treeprododd}
If $T$ and $T'$ are trees which are non-$3$-peripheral with $|T|,|T'| \geq 2$ and $\diam(T \square T')$ is odd, then $\aw(T \square T', 3) = 4$.  
\end{thm}

 Corollary \ref{d(u,v)+1} says that the eccentricity of any vertex in a tree is realized by a peripheral vertex. However, it is stated in a way that is easier to use in practice.

\begin{cor}[\cite{BMSWW}]\label{d(u,v)+1}
If $u$ is not a peripheral vertex of $T$ and $v \in V(T)$, then there exists a vertex $w \in V(T)$ such that $\dist(w,v) = \dist(u,v) + 1$. 
\end{cor}

\begin{thm}[Jordan 1869]\label{thm:Jordan} The center of a tree consists of one vertex or two vertices. 
\end{thm}

In general, the center of a graph need not be connected, and in fact, components of the center can be arbitrarily far apart.  Consider, for example, a $C_6$ with a leaf on every other vertex.  For trees, this is not the case.  We include the following useful observation and its proof for completeness.

\begin{obs}\label{obs:diam-paths-int-center}
    All diameter paths in a tree $T$ intersect every central vertex of $T$. In particular, all diameter paths in a tree $T$ intersect.
\end{obs}

\begin{proof}
    Suppose we have a diameter path $P$ with peripheral vertices $u$ and $v$. Define $C_T$ to be the center of $T$ and $C_P$ to be the center of $P$ as a subgraph of $T$. Let $w \in C_P$ and assume $w \notin C_T$. Since $\rad(T) \in \left\{\frac{\diam(T)}{2},\frac{\diam(T)+1}{2}\right\}$, there exists some $x \in V(T)$ such that $d(w,x) >\frac{\diam(T)+1}{2}$. But now, either $d(u,x) = d(u,w) + d(w,x)$ or $d(v,x) = d(v,w) + d(w,x)$. Since $d(v,w),d(u,w) \geq \frac{\diam(T)-1}{2}$, either case gives a distance larger than the diameter, a contradiction.

    Thus, $C_P \subseteq C_T$. Theorem \ref{thm:Jordan} implies the center of our tree is one or two vertices depending on the parity of the diameter. This means $|C_P| = |C_T|$ in either case. So, $C_P = C_T$, our desired result.
\end{proof}

The structure given by Observation \ref{obs:diam-paths-int-center} allows us to give the proof of the following straightforward result.

\begin{lem}\label{general tree peripheral lem}
    Suppose $T$ is a tree with $u,v \in V(T)$ which realize the diameter. If $x$ is a peripheral vertex of $T$, then $x$ is diameter away from $u$ or diameter away from $v$.
\end{lem}

\begin{proof} 
    If $x$ is diameter away from $u$ or $v$ we are done so assume $d(x,u)\neq \diam(T)$ and $d(x,v)\neq \diam(T)$.  Since $x$ is peripheral, there exists some $y\notin \{u,v\}$ such that $d(x,y) = \diam(T)$. Define $P$ and $P'$ to be the $u-v$ and $x-y$ paths in $T$, respectively.  By Observation \ref{obs:diam-paths-int-center} we know that $P$ and $P'$ intersect so define $s$ and $t$ to be the vertices in $V(P) \cap V(P')$ closest to $x$ and $y$, respectively. 

    Without loss of generality, suppose $d(u,s) \leq d(u,t)$ (see Figure \ref{fig:lem2.15}). If $d(x,s) < d(u,s)$, then a contradiction follows from
    \[
    \begin{split}
        \diam(T) & = d(x,y) \\
        & = d(x,s) + d(s,t) + d(t,y) \\
        & < d(u,s) + d(s,t) + d(t,y) \\
        & = d(u,y).
    \end{split}
    \]
    Similarly, if $d(u,s) < d(x,s)$ we get the contradction that $\diam(T) < d(x,v)$.
    Thus, $d(u,s) = d(x,s)$ which implies that $d(x,v) = \diam(T)$.\end{proof}

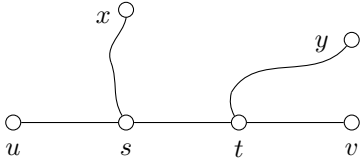
\begin{figure}
    \centering
    \begin{tikzpicture}
        \node[draw,circle,inner sep = 2] (11) at (0,4) {};
        \node[draw,circle,inner sep = 2] (ii) at (1.5,4) {};
        \node[draw,circle,inner sep = 2] (jj) at (3,4) {};
        \node[draw,circle,inner sep = 2] (kk) at (4.5,4) {};
        \node[draw,circle,inner sep = 2] (vv) at (4.5,5.1) {};
        
        \draw (11) -- (ii);
        \draw (ii) -- (jj);
        \draw (jj) -- (kk);
        \draw (ii) to[in = -70, out = 120] (1.3,4.8);
        \draw (1.3,4.8) to[in = -80, out = 110] (1.5,5.5);
        \draw[bend left = 20] (jj) to (2.9,4.4);
        \draw (2.9,4.4) to[in = -130, out = 60] (vv);

        \node[draw,circle,inner sep = 2,fill=white] (uu) at (1.5,5.5) {};
        \node at (0,3.65) {$u$};
        \node at (1.5,3.65) {$s$};
        \node at (3,3.65) {$t$};
        \node at (4.5,3.65) {$v$};
        \node at (1.2,5.4) {$x$};
        \node at (4.1,5.05) {$y$};
    \end{tikzpicture}
    \caption{Relationship between peripheral vertices $u,x,y$ and $v$  when $d(u,s) \le d(u,t)$, as in Lemma \ref{general tree peripheral lem}.}
    \label{fig:lem2.15}
\end{figure}

\begin{lem}[\cite{BMSWW}]\label{lem:3.10}
    Suppose $T$ is a non-$3$-peripheral tree with and $u_i,u_j\in V(T)$ realize the diameter of $T$. If there exist $u_x,u_y \in V(T)$ such that $\dist(u_x,u_j) = \diam(T)$ and $\dist(u_i,u_y) = \diam(T)$, then $\dist(u_x,u_y) = \diam(T)$. 
\end{lem}

Note that the four vertices in Lemma \ref{lem:3.10} need not be distinct to apply the result, a fact that is used regularly in this paper. 

\begin{lem}[\cite{BMSWW}]\label{lem:3-peripheral even diam}
    If $T$ is $3$-peripheral, then $\diam(T)$ is even.  Further, for any three vertices that are pairwise distance $\diam(T)$ apart, there is some vertex that is equidistant from all three of them.
\end{lem}

In Section $5$ of \cite{BMSWW}, it was found that when classifying $\aw(T\square T',3)$ when $\diam(T\square T')$ is even, the partition of trees into $3$-peripheral and non-$3$-peripheral was insufficient. 
To refine the partition further some new definitions are needed.

    If $T$ is a tree with peripheral vertex $v$, we define $T_{v^-}$ to be the tree obtained from $T$ by removing all vertices which realize the diameter of $T$ with $v$. 
    If $u$ is any vertex of $T$, we define $T_{u^+}$ to be the tree $T$ with an additional leaf adjacent to $u$.

    One motivation for the $T_{v^-}$ definition is that when $v$ is peripheral we change the parity of the diameter of $T$.  
    In particular, we want to use Lemma \ref{lem:3.10} with leaves that are $\diam(T) - 1$ away from each other. By moving to the subgraph $T_{v^-}$, Observation \ref{obs:diamTtilde} allows us to apply the desired lemma.

\begin{obs}\label{obs:diamTtilde}
    For any peripheral vertex $v$ in a tree $T$, we have \[\diam\left(T_{v^-}\right) = \diam(T)-1.\]
\end{obs}

\begin{proof}
    Since $\epsilon_{T_{v^-}}(v) = \diam(T)-1$, we certainly have $\diam\left(T_{v^-}\right) \geq \diam(T)-1$. To show the other inequality, we will show that all diameter paths in $T$ lose a vertex in $T_{v^-}$. If $x$ and $y$ realize the diameter of $T$, Lemma \ref{general tree peripheral lem} implies that either $x$ or $y$ is diameter away from $v$ meaning that one of them will not appear in $T_{v^-}$. Thus, $T_{v^-}$ has no geodesics of length $\diam(T)$, our desired result.
\end{proof}

Since $T_{v^-}$ is an isometric subgraph of $T$ and $T$ is an isometric subgraph of $T_{v^+}$, we will often use the notation $d_T(x,y)$ interchangeably with $d_{T_{v^-}}(x,y)$ or $d_{T_{v^+}}(x,y)$ since these quantities are equal provided $x$ and $y$ are in each of the necessary trees.

\begin{defn}\label{defn:weak/strong non-3-per}
    Let $T$ be a nontrivial tree.
    \begin{enumerate}[(i)]
        \item Let $T$ be non-$3$-peripheral with odd diameter. We say $T$ is \emph{strongly non-$3$-peripheral} if there exists a peripheral vertex $v$ such that $T_{v^-}$ is non-$3$-peripheral. Otherwise, we say $T$ is \emph{weakly non-$3$-peripheral}. That is, for all peripheral vertices $v$ of $T$, $T_{v^-}$ is $3$-peripheral.

        \item Let $T$ be  non-$3$-peripheral with even diameter. We say $T$ is \emph{strongly non-$3$-peripheral} if for all $v \in V(T)$, $T_{v^+}$ is non-$3$-peripheral. Otherwise, we say it is \emph{weakly non-$3$-peripheral}. That is, there exists some $v \in V(T)$ such that $T_{v^+}$ is $3$-peripheral.

    \end{enumerate}
\end{defn}

\begin{ex} \rm{ Recall from Figure \ref{fig:cex} that the graph $H$ is weakly non-$3$-peripheral with odd diameter. Using the same vertex labeling as the figure, this can be seen because $T_{v_1^-}$ is $3$-peripheral. However, $P_4$ is strongly non-$3$-peripheral of odd diameter because the removal of any peripheral vertex yields a $P_3$ which is not $3$-peripheral. 

An example of a weakly non-$3$-peripheral tree with even diameter is $P_3$. This can be seen because if $c$ is the central vertex of $P_3$, then $T_{c^+}$ is isomorphic to the star $K_{1,3}$ which is $3$-peripheral. However, $P_5$ is strongly non-$3$-peripheral with even diameter via Lemma \ref{lem:weakly/per}. Specifically, $P_5$ has no vertex which is $\diam(P_5)-1$ away from both it's peripheral vertices.}

\end{ex}

Recall that Observation \ref{obs:diamTtilde} states that $T_{v^-}$ operation lowers the diameter by $1$ when applied to a peripheral vertex. This is important because we would like to be able to achieve a $3$-peripheral graph with this. Since this can only be done if $T$ has even diameter (as seen in Lemma \ref{lem:3-peripheral even diam}), it is important to alter the diameter in some fashion. However, when using the $T_{v^+}$ operation to an even diameter vertex, it is important that we do not change the diameter.
Since $T_{v^+}$ only changes the diameter when applied to a peripheral vertex, we never use it on a peripheral vertex.
As we will find in Lemma \ref{lem:weakly/per}, it is useful to apply $T_{v^+}$ to a vertex whose eccentricity is one less than the diameter of our tree.

\begin{lem}\label{lem:weakly/per}
    If $T$ is a weakly non-$3$-peripheral tree with even diameter, say $T_{u^+}$ is $3$-peripheral, then for any peripheral vertex $v$ of $T$, $d(u,v)=\diam(T)-1$.
\end{lem}

\begin{proof}
    Let $u'$ be the added leaf to $u$ in $T_{u^+}$. 
    Since $T_{u^+}$ is $3$-peripheral while $T$ is not, there exist peripheral vertices $v_j$ and $v_k$ in $T$ which realize the diameter of $T$ with each other and with $u'$.  
    If $v_i$ and $v_j$ are the only peripheral vertices of $T$, then we are done. 
    So suppose $v$ is any other peripheral vertex of $T$. 
    Since $T$ is non-$3$-peripheral, either $d(v,v_j)$ or $d(v,v_k)$ is less than $\diam(T)$. 
    Without loss of generality, suppose $d(v,v_j) < \diam(T)$. 
    Since $u'$ and $v_j$ realize the diameter in $T_{u^+}$, applying Lemma \ref{general tree peripheral lem} gives $d_{T_{u^+}}(v,u') = \diam(T_{u^+}) = \diam(T)$. 
    This means $d(v,u) = \diam(T)-1$.
\end{proof}

\section{Strongly and weakly peripheral trees}\label{sec:3}

In this section, we prove our main theorem below which classifies the anti-van der Waerden number of all products of trees when the diameter is even.

\begin{customthm}{1.3}
 Let $T$ and $T'$ be nontrivial trees, where $\diam(T\square T')$ is even. Then 
    \[\aw(T\square T',3) = 
\begin{cases}
    3 & \text{if $T$ or $T'$ is weakly non-$3$-peripheral or isomorphic to $P_2$}, \\

    4 & \text{if $T$ and $T'$ are both strongly non-$3$-peripheral and not $P_2$}.
\end{cases}\]

\end{customthm}

Recall that Theorem \ref{theorem:rsw} says this number will be either three or four. In \cite{BMSWW}, this has already been done when one of the trees is $3$-peripheral in which we get an anti-van der Waerden number of three. As for the remaining cases, the language developed allows us to succinctly categorize the two different possibilities. If both trees are strongly non-$3$-peripheral, then the anti-van der Waerden number is $4$, and if either tree is weakly non-$3$-peripheral, then the anti-van der Waerden number is $3$. These two cases will be split further into the cases of whether our trees both have odd diameter or both have diameter. These four cases can be seen in Propositions \ref{prop:strongly/odd}, \ref{prop:strongly/even}, \ref{prop:weakly/odd}, and \ref{prop:weakly/even}.

The first step in this process is handling when the anti-van der Waerden number is four, which as we stated is when both trees are strongly non-$3$-peripheral. To achieve this, Theorem \ref{theorem:rsw} implies that it suffices to provide a $3$-coloring that is rainbow $3$-AP free. Lemma \ref{lem:even-diam-rb-free-coloring} will provide such coloring, however, this classification differs depending on the parity of the trees' diameters and the argument is separated into Propositions \ref{prop:strongly/odd} and \ref{prop:strongly/even}.

The intuition of why these definitions are important can be seen by exploring past papers. In \cite{SWY,BMSWW,DMS,MW,RSW,finabgroup}, rainbow-free colorings are often constructed using $red,blue,green$ where one color is dominantly used ($green$ in Section \ref{sec:3}) and a small number of vertices are colored with $red$ and $blue$. 
The intuition is that in order to avoid a rainbow $3$-AP, we would like the distance from any $red$ vertex to any $blue$ vertex to be odd and large. 
If the distance were even, then  it is possible for a path with $red$ and $blue$ endpoints to have a midpoint colored $green$ giving us a rainbow 3-AP of the form $(red,green,blue)$. , and if the distance were small enough, then it may be possible to construct a rainbow 3-AP of the form $(red,blue,green)$ or $(blue,red,green)$. 
Therefore, our rainbow-3AP-free coloring in Lemma~\ref{lem:even-diam-rb-free-coloring} was constructed by making the distance between any $red$ and $blue$ vertices as large as possible, i.e. $\diam(T\square T')-1$ since $\diam(T\square T')$ is even.
This was the motivation to turn to subgraphs of $T$ and $T'$ whose diameter is exactly $1$ less than diameter of $T$ or diameter of $T'$. 
While not obvious, it turns out that the correct way to achieve this is with $T_{v^-}$ when the diameter of $T$ is odd and $T_{v^+}$ when the diameter of $T$ is even. 
Observation \ref{obs:diamTtilde} explains that performing the $T_{v^-}$ operation on a tree will lower the diameter by $1$.  If $red$ and $blue$ vertices have distance $\diam(T\square T')-1$, then one must be peripheral. Lemma \ref{lem:weakly/per} now gives a specific and exploitable vertex which is $\diam(T)-1$ from this $red$ or $blue$ vertex.

\begin{lem}\label{lem:even-diam-rb-free-coloring}
    Suppose $T,T'$ are nontrivial trees which are non-$3$-peripheral such that $\diam(T\square T')$ is even. 
    Suppose $v_{1,1}$ and $v_{j,k}$ realize the diameter of $T \square T'$ such that $T_{u_1^-}$ and $T'_{w_1^-}$ are non-$3$-peripheral. 
    Define $c: V(T\square T') \to \{red,blue,green\}$ as follows
    \[c(v_{a,b}) = 
    \begin{cases}
        blue & \text{if $d(v_{a,b},v_{1,1}) = \diam(T\square T') - 1$,} \\
        red & \text{if $d(v_{a,b},v_{j,k}) = \diam(T \square T')$,} \\
        green & \text{otherwise.}
    \end{cases}\]
    Then we have the following:
    \begin{enumerate}[$(i)$]
        \item $c$ is well-defined
        \item\label{3ii} if $c(x)=red$ and $c(y) = blue$, then $d(x,y) = \diam(T \square T')-1$
        \item if $c(x) = red$ and $d(x,y) = \diam(T\square T') - 1$, then $c(y) = blue$.
        \item if $(x,y,z)$ is a rainbow $3$-AP, then $c(y) = blue$.
    \end{enumerate}
\end{lem}

\begin{proof}
    \textbf{(i)} Note that by the definition of $c$, no vertex will be $blue$ and $green$, and no vertex will be $red$ and $green$.  Thus, we only need to check if a vertex will be colored both $red$ and $blue$.
    For the sake of contradiction, assume $v_{a,b}$ is such a vertex, that is, $d(v_{a,b},v_{1,1}) = \diam(T\square T') - 1$ and $d(v_{a,b},v_{j,k}) = \diam(T \square T')$. 
    Since $d(v_{a,b},v_{1,1}) = \diam(T\square T') - 1$, either $d_T(u_a,u_1)=\diam(T)$ or $d_{T'}(w_b,w_1)=\diam(T')$.
    Further, $d(v_{a,b},v_{j,k}) = \diam(T \square T')$ implies
    and $d_T(u_a,u_j) = \diam(T)$ and $d_{T'}(w_b,w_k) = \diam(T')$.
    In either case, either $T$ or $T'$ is $3$-peripheral, a contradiction.
    Thus, $c$ is well-defined.
    
    \textbf{(ii)} Suppose $c(x) = red$ with $x = v_{x_1,x_2}$ and $c(y) = blue$ with $y = v_{y_1,y_2}$.  
    Since $d(v_{y_1,y_2},v_{1,1}) = \diam(T\square T') - 1$, it follows that $d_T(u_{y_1},u_1) = \diam(T)$ and $d_{T'}(w_1,w_{y_2}) = \diam(T')-1$, or that $d_T(u_{y_1},u_1) = \diam(T) - 1$ and $d_{T'}(w_1,w_{y_2}) = \diam(T')$.

    First, suppose that \[d_T(u_{y_1},u_1) = \diam(T) \text{ and } d_{T'}(w_1,w_{y_2}) = \diam(T')-1.\] 
    Since $d(v_{x_1,x_2},v_{j,k}) = \diam(T \square T')$, it follows that $d_T(u_{x_1},u_j) = \diam(T)$ and $d_{T'}(w_{x_2},w_k) = \diam(T')$. 
    So, Lemma \ref{lem:3.10} implies that $d_T(u_{x_1},u_{y_1}) = \diam(T)$. 
    
    It remains to show $d(w_{x_2},w_{y_2})=\diam(T')-1$. 
    We first rule out some trivial cases. 
    If $w_{x_2}=w_1$, then we immediately have our desired result. 
    Notice $w_{x_2} \neq w_k$ and $w_1 \neq w_k$ since both pairs realize the diameter.      
    That is, we can suppose $w_1,w_k$ and $w_{x_2}$ are distinct. 
    Let $w_{k-1}$ be the unique neighbor of $w_k$, and note that Observation \ref{obs:diamTtilde} implies that $ \diam(T'_{w_1^-})= \diam(T')-1= d_{T'}(w_{x_2},w_{k-1})$.
    Additionally, $d_{T'}(w_{x_2},w_1) \neq \diam(T')$, otherwise $w_1,w_k,w_{x_2}$ are pairwise diameter apart, contradicting that $T'$ is non-$3$-peripheral. 
    This implies $w_{x_2} \in V(T'_{w_1^-})$.
    Recall that $d_{T'}(w_1,w_{y_2}) = \diam(T') - 1 = \diam(T'_{w_1^-})$ implying that $w_{y_2} \in V(T'_{w_1^-})$.  
    Since $T'_{w_1^-}$ is non-$3$-peripheral, Lemma \ref{lem:3.10} can be applied to vertices $w_1,w_{k-1},w_{x_2}$ and $w_{y_2}$ in $T'_{w_1^-}$ to get 
    \[d_{T'}(w_{x_2},w_{y_2}) = d_{T'_{w_1^-}}(w_{x_2},w_{y_2}) = \diam(T'_{w_1^-}) = \diam(T')-1.\] 
    Thus, \[d(x,y) = d_T(u_{x_1},u_{y_1}) + d_{T'}(w_{x_2},w_{y_2}) = \diam(T) + \diam(T') - 1 = \diam(T \square T')-1.\] 

     Second, if we instead suppose \[d_T(u_{y_1},u_1) = \diam(T)-1 \text{ and }d_{T'}(w_1,w_{y_2}) = \diam(T'),\] then a similar argument shows that $u_1,u_j$ and $u_{x_1}$ are distinct and that applying Lemma \ref{lem:3.10} to $w_1,w_k,w_{x_2}$ and $w_{y_2}$ in $T'$ and to $u_1,u_{j-1},u_{x_1}$ and $u_{y_1}$ in $T_{u_1^-}$ yields the same result. 

    \textbf{(iii)} Since $c(x) = red$, we have $d(x,v_{j,k}) = \diam(T\square T')$ implying that $d_T(u_{x_1},u_j) = \diam(T)$ and $d_{T'}(w_{x_2}, w_k) = \diam(T')$. 
    Since $d(x,y) = \diam(T\square T')-1$, we have that $d_T(u_{x_1},u_{y_1}) = \diam(T)$ and $d_{T'}(w_{x_2},w_{y_2}) = \diam(T')-1$ or that $d_T(u_{x_1},u_{y_1}) = \diam(T)-1$ and $d_{T'}(w_{x_2},w_{y_2}) = \diam(T')$. 
    First, suppose $d_T(u_{x_1},u_{y_1}) = \diam(T)$ and $d_{T'}(w_{x_2},w_{y_2}) = \diam(T')-1$. 
    Notice that since \[d_T(u_1,u_j) = d_T(u_{x_1},u_j) = d_T(u_{x_1},u_{y_1}) = \diam(T),\] Lemma \ref{lem:3.10} implies that $d_T(u_1,u_{y_1}) = \diam(T)$. 
    
    It remains to show $d(w_1,w_{y_2})=\diam(T')-1$. 
    If $w_{x_2}=w_1$, then we immediately have our desired result. 
    Notice $w_{x_2} \neq w_k$ and $w_1 \neq w_k$ since both pairs realize the diameter. 
    Thus we can suppose $w_1,w_k$ and $w_{x_2}$ are distinct.
    It now follows that $d_{T'}(w_{x_2},w_{k-1})=\diam(T')-1$. 
    Notice that \[d_{T'}(w_1,w_{k-1}) = d_{T'}(w_{x_2},w_{k-1}) = d_{T'}(w_{x_2},w_{y_2}) = \diam(T')-1 = \diam(T'_{w_1^0}).\] 
    Since $T'_{w_1^-}$ is non-$3$-peripheral, we now show $w_{k-1},w_{x_2},w_1, w_{y_2} \in V(T'_{w_1^-})$ so that Lemma \ref{lem:3.10} applied to $T'_{w_1^-}$ gives $d_{T'}(w_1,w_{y_2}) = \diam(T')-1$. 
    Since $d_{T'}(w_1,w_k) =\diam(T')$, we have $w_1,w_{k-1} \in V(T'_{w_1^-})$. 
    If $w_{x_2} \notin V(T'_{w_1^-})$, then $d(w_1,w_{x_2})=\diam(T')$ and $w_1,w_k$ and $w_{x_2}$ pairwise realize the diameter, contradicting that $T'$ is non-$3$-peripheral. 
    If $w_{y_2} \notin V(T'_{w_1^-})$, then $d_{T'}(w_1,w_{y_2})=\diam(T')$, then combining Lemma \ref{lem:3.10} and that \[d_{T'}(w_1,w_{y_2}) = d_{T'}(w_1,w_k) = d_{T'}(w_{x_2},w_k) = \diam(T')\] shows that $d_{T'}(w_{x_2},w_{y_2})= \diam(T')$, a contradiction. 
    Thus, $d_{T'_{w_1^-}}(w_1,w_{y_2}) = d_{T'}(w_1,w_{y_2}) = \diam(T')-1$, as desired. 
    Finally, \[d(v_{1,1},y) = d_T(u_1,u_{y_1}) + d_{T'}(w_1,w_{y_2}) = \diam(T\square T')-1,\] showing that $c(y)=blue$. 

    Alternatively, if $d_T(u_{x_1},u_{y_1}) = \diam(T)-1$ and $d_{T'}(w_{x_2},w_{y_2}) = \diam(T')$, then a similar argument shows that $u_1,u_j$ and $u_{x_1}$ are distinct and that applying Lemma \ref{lem:3.10} to $w_1,w_k,w_{x_2}$ and $w_{y_2}$ in $T'$ and to $u_1,u_{j-1},u_{x_2}$ and $u_{y_2}$ in $T_{u_1^-}$ yields the same result.

    \textbf{(iv)} Because of the symmetry of $3$-APs, any rainbow $3$-AP can be classified by the color of the middle vertex. Suppose $(x,y,z)$ is a rainbow $3$-AP. First, assume $c(y)=green$. Then $x$ and $z$ are colored $red$ and $blue$ in some order. 
    Recall that $T \square T'$ is bipartite. Since $\diam(T\square T')$ is even, (ii) implies that $x$ and $z$ are in different partite sets. 
    Thus, no such $y$ can be an equal distance from $x$ and $z$, contradicting that $(x,y,z)$ is a $3$-AP. 
    Second, assume $c(y)=red$. 
    Since one of $x$ or $z$ is $blue$, (ii) implies the common distance of our $3$-AP is $\diam(T\square T')-1$. But (iii) implies any vertex distance $\diam(T\square T')-1$ from $y$ has color $blue$, contradicting that one of $x$ or $z$ is $green$. Thus, $c(y) = blue$.
\end{proof}

The remainder of this section completes the proof of the main theorem, showing that when $\diam(T\square T')$ is even, we have \[\aw(T\square T',3) = 
\begin{cases}
    3 & \text{if $T$ or $T'$ is weakly non-$3$-peripheral or isomorphic to $P_2$}, \\
    4 & \text{if $T$ and $T'$ are both strongly non-$3$-peripheral and not $P_2$}.
\end{cases}\]
This will be broken into four cases depending on whether the trees have odd or even diameter and on weakly and strongly non-$3$-peripheral properties of trees. The case when one of the trees is $P_2$ is quite different, since $(P_2)_{u^-}$ is an isolated vertex. Therefore, we separate this case from the remaining trees in Lemma \ref{lem:annoyingP2}.

 We provide the motivation for two of the four cases below. First, if $T$ is a strongly non-$3$-peripheral tree of odd diameter, then for some vertex $u$, $T_{u^-}$ is non-$3$-peripheral and Lemma \ref{lem:even-diam-rb-free-coloring} gives a rainbow-free $3$-coloring of $T_{u^-} \square T'$. By carefully extending this coloring to $T \square T'$, we can avoid rainbow $3$-APs. Applying Theorem \ref{theorem:rsw}, this means the anti-van der Waerden number is four. 
Second, if $T$ is a weakly non-$3$-peripheral tree of even diameter, consider an arbitrary $3$-coloring of $T \square T'$. This can be extended to a $3$-coloring of $T_{v^+} \square T'$ where $T_{v^+}$ is $3$-peripheral. Then Theorem \ref{TxG 3-per} says this coloring admits a $3$-AP. If we choose our extension coloring carefully, we can guarantee that this $3$-AP is in $T \square T'$, too. The formal proof the authors provide for Proposition \ref{prop:weakly/even} is a slight variation of this idea. 
It is based on the fact that any $3$-peripheral tree with all its leaves removed will remain $3$-peripheral or will be a single vertex. 
The proof looks for a $3$-AP in this subgraph rather than in $T_{v^+} \square T'$. 
That way, we are guaranteed that any $3$-AP we find will also appear in the parent graph.  While the remaining two cases are less intuitive, they worked out as desired (see Propositions \ref{prop:strongly/even} and \ref{prop:weakly/odd}). 

We begin showing that if $T$ and $T'$ are strongly non-$3$-peripheral trees, where neither are $P_2$, then their product $T \square T'$ has anti-van der Waerden number $4$.  Before we begin Proposition \ref{prop:strongly/odd}, as a reminder, a tree $T$ with odd diameter is called strongly non-$3$-peripheral if the tree $T_{v^-}$ is non-$3$-peripheral for every peripheral $v \in V(T)$.

\begin{prop}\label{prop:strongly/odd}
    Suppose $T$ and $T'$ are strongly non-$3$-peripheral trees with odd diameter of at least $3$. Then, $\aw(T\square T',3) = 4$. 
\end{prop}

\begin{proof}
    Since $T$ and $T'$ are strongly non-$3$-peripheral, there exist peripheral vertices $u_1 \in V(T)$ and $w_1 \in V(T')$ such that $T_{u_1^-}$ and $T'_{w_1^-}$ are non-$3$-peripheral. 
    Since $u_1$ and $w_1$ are peripheral, there exist $u_j \in V(T)$ and $w_k \in V(T')$ such that $d_T(u_1,u_j) = \diam(T)$ and $d_{T'}(w_1,w_k) = \diam(T')$. 
    Now color $T\square T'$ using the coloring from Lemma \ref{lem:even-diam-rb-free-coloring} with $v_{1,1}$ and $v_{j,k}$ playing the same role as in the lemma. 
    If $(x,y,z)$ is a rainbow $3$-AP, then Lemma \ref{lem:even-diam-rb-free-coloring}(ii) and (iv) give that the common difference is $\diam(T\square T') - 1$ and, without loss of generality, that $c(x) = green$, $c(y) = blue$ and $c(z) = red$.
    Suppose $x=v_{x_1,x_2}$ and $y=v_{y_1,y_2}$.
    
    Suppose that $x=v_{x_1,x_2}$ and $y=v_{y_1,y_2}$. 
    Since $c(y) = blue$, we have that $d(y,v_{1,1}) = \diam(T\square T') - 1$ implying that either both 
    
\begin{equation}\label{eq1}d_T(u_1,u_{y_1}) = \diam(T) \text{  and  } d_{T'}(w_1,w_{y_2}) = \diam(T')-1\end{equation}
or

\begin{equation}\label{eq2}
d_T(u_1,u_{y_1}) = \diam(T) - 1 \text{  and  } d_{T'}(w_1,w_{y_2}) = \diam(T').\end{equation}

Additionally, since $d(x,y) = \diam(T\square T') - 1$, it follows that

\begin{equation}\label{eq4}d_T(u_{x_1},u_{y_1}) = \diam(T) - 1 \text{  and  } d_{T'}(w_{x_2},w_{y_2}) = \diam(T')\end{equation}

or

\begin{equation}\label{eq3}d_T(u_{x_1},u_{y_1}) = \diam(T)  \text{  and  } d_{T'}(w_{x_2},w_{y_2}) = \diam(T')-1. \end{equation}

    First suppose that Equations from \ref{eq1} and \ref{eq4} hold.
    Recall that $d_{T'}(w_1,w_k) = \diam(T')$. 
    Since $d_{T'}(w_1,w_{y_2})$ $< \diam(T')$ and $w_{y_2}$ is a peripheral vertex of $T'$, Lemma \ref{general tree peripheral lem} implies that $d_{T'}(w_{y_2},w_k) = \diam(T')$.
    Let $w_{k-1}$ be the neighbor of $w_k$ in $T'$ and notice that $w_{k-1}, w_{y_2}$ and $w_1$ are vertices in $T'_{w_1^-}$.  
    Additionally, Observation \ref{obs:diamTtilde} implies $\diam(T'_{w_1^-}) = \diam(T') -1$.
    Thus, 
    \[\diam(T'_{w_1^-}) = d_{T'}(w_1,w_{y_2}) = d_{T'}(w_1,w_{k-1}) = d_{T'}(w_{y_2},w_{k-1}).\] 
    Furthermore, these three vertices are distinct since otherwise we have that $\diam(T')=1<3$.
    Thus, $T'_{w_1^-}$ is $3$-peripheral, a contradiction. 

    Second, suppose that Equations \ref{eq1} and \ref{eq3} hold. 
    Since $d_T(u_1,u_j) = \diam(T)$, Lemma \ref{lem:3.10} implies that $d_T(u_{x_1},u_j) = \diam(T)$. 
  
    Consider the case where $d_{T'}(w_1,w_{x_2}) = \diam(T')$. Let $P$ and $P'$ represent the $w_1-w_{y_2}$ and $w_{x_2}-w_{y_2}$ paths, respectively. 
    Let $w_\ell \in V(T')$ be the vertex in $V(P) \cap V(P')$ nearest to $w_1$. 
    Then 
        \[\diam(T')-1  = d_{T'}(w_1,w_{y_2}) = d_{T'}(w_1,w_\ell) + d_{T'}(w_\ell,w_{y_2}), \] 
        \[\diam(T')-1  = d_{T'}(w_{x_2},w_{y_2}) = d_{T'}(w_{x_2},w_\ell) + d_{T'}(w_\ell,w_{y_2}), \text{ and} \]
        \[\diam(T')  = d_{T'}(w_1,w_{x_2}) = d_{T'}(w_1,w_\ell) + d_{T'}(w_\ell,w_{x_2}).\]
    The first two equations imply that $d_{T'}(w_1,w_\ell) = d_{T'}(w_{x_2},w_\ell)$ and the third implies that both distances equal $\diam(T')/2$. However, $\diam(T')$ is odd, a contradiction. 
    Thus, $d_{T'}(w_1,w_{x_2})<\diam(T')$. It follows that $w_{x_2} \in V(T'_{w_1^-})$ and applying Lemma \ref{lem:3.10} to $w_1$,$w_{y_2}$,$w_{x_2}$,$w_{k-1}$ and graph $T'_{w_1^-}$ gives $d_{T'}(w_{x_2}, w_{k-1}) = \diam(T')-1$.
    If $w_{x_2}=w_k$, then the previous sentence implies $\diam(T')$ is even, a contradiction. 
    This means $w_{x_2} \neq w_k$ so the $w_{x_2} - w_k$ path contains $w_{k-1}$ and we can conclude $d_{T'}(w_{x_2},w_k) = \diam(T')$.
    Finally, we have \[d(x,v_{j,k}) = d_T(u_{x_1},u_j) + d_{T'}(w_{x_2},w_k) = \diam(T) + \diam(T') = \diam(T\square T'),\] showing that $c(x)=red$, contradicting the assumption that $c(x) = green$.

    The case where the Equations from \ref{eq2} and \ref{eq3} hold has an argument similar to the case where Equations \ref{eq1} and \ref{eq4} hold.
     The case where the Equations from \ref{eq2} and \ref{eq4} hold has an argument similar to the case where Equations \ref{eq1} and \ref{eq3} hold.
    \end{proof}

    Recall that non-$3$-peripheral tree $T$ with even diameter is strongly non-$3$-peripheral if $T_{v^+}$ is non-$3$-peripheral for all $v \in V(T)$.

\begin{prop}\label{prop:strongly/even}
    Suppose $T$ and $T'$ are nontrivial, strongly non-$3$-peripheral trees with even diameter. Then $\aw(T\square T',3) = 4$. 
\end{prop}

\begin{proof}
    First notice that Lemmas \ref{lem:3-peripheral even diam} and \ref{obs:diamTtilde} imply that any peripheral vertices $u_1\in V(T)$ and $w_1\in V(T')$ have the property that $T_{u_1^-}$ and $T'_{w_1^-}$ are non-$3$-peripheral. 
    Let $c$ be the coloring as in Lemma \ref{lem:even-diam-rb-free-coloring} and let $v_{1,1},v_{j,k} \in V(T\square T')$ realize the diameter of $T\square T'$ be defined as in the previous two arguments. 
   
    For the sake of contradiction, assume $(x,y,z)$ is a rainbow $3$-AP.
    Lemma \ref{lem:even-diam-rb-free-coloring}(ii) and (iv) now gives that the common difference is $\diam(T\square T') - 1$ and, without loss of generality, that $c(x) = green$, $c(y) = blue$ and $c(z) = red$.
    Suppose $x=v_{x_1,x_2}$ and $y=v_{y_1,y_2}$.  Since $c(y) = blue$, we have $d(y,v_{1,1}) = \diam(T\square T') - 1$ implying that
    
   \begin{equation}\label{eq5}d_T(u_1,u_{y_1}) = \diam(T) \text{ and } d_{T'}(w_1,w_{y_2}) = \diam(T')-1\end{equation}
   
   or that  
   
   \begin{equation}\label{eq6}d_T(u_1,u_{y_1}) = \diam(T) - 1 \text{ and } d_{T'}(w_1,w_{y_2}) = \diam(T').\end{equation} 
   
    Additionally, since $d(x,y) = \diam(T\square T') - 1$, it follows that 
    
    \begin{equation}\label{eq7} d_T(u_{x_1},u_{y_1}) = \diam(T)  \text{ and } d_{T'}(w_{x_2},w_{y_2}) = \diam(T')-1\end{equation} or that \begin{equation}\label{eq8}d_T(u_{x_1},u_{y_1}) = \diam(T) - 1  \text{ and } d_{T'}(w_{x_2},w_{y_2}) = \diam(T').\end{equation}

    First, suppose that Equations \ref{eq5} and \ref{eq7} hold.  Lemma \ref{lem:3.10} applied to $u_1$, $u_{y_1}$, $u_j$ and $u_{x_1}$ implies that $d_{T}(u_{x_1},u_j)=\diam(T)$.
    
    Assume $d(w_1,w_{x_2}) = \diam(T')$ and
    recall that $d_{T'}(w_1,w_{y_2}) = \diam(T')-1$ and $d_{T'}(w_{x_2},w_{y_2}) = \diam(T')-1$. 
    Notice that $w_{y_2} \notin \{w_1,w_{x_2}\}$ otherwise $\diam(T')$ is odd, a contradiction. 
    Thus, $T'_{w_{y_2}^+}$ is $3$-peripheral, contradicting that $T'$ is strongly non-$3$-peripheral.
    This means we must have $d(w_1,w_{x_2}) < \diam(T')$.
    Let $w_{k-1}$ the neighbor of $w_k$ in $T'$.  Applying Lemma \ref{lem:3.10} to $w_1,w_{y_2},w_{x_2},w_{k-1} \in V(T'_{w_1^-})$ gives $d_{T'}(w_{x_2},w_{k-1})=\diam(T')-1$. 
    If $w_{x_2}=w_k$, then this means $\diam(T')=2$. So, $T'$ is a star with at least two leaves. 
    However, no such star is strongly non-$3$-peripheral, a contradiction. 
    Thus, $w_{x_2} \neq w_k$ so the $w_{x_2}-w_k$ path contains $w_{k-1}$ and we can conclude $d_{T'}(w_{x_2},w_k)=\diam(T')$.
    Finally, $d(x,v_{j,k})=\diam(T\square T')$ so $c(x) = red$ which contradicts that $c(x) = green$.
    
    Second, assume that Equations \ref{eq5} and \ref{eq8} hold. 
    Recall that $d_{T'}(w_1,w_k) = \diam(T')$. 
    Since $d_{T'}(w_1,w_{y_2}) < \diam(T')$ and $w_{y_2}$ is a peripheral vertex, Lemma \ref{general tree peripheral lem} implies that $d_{T'}(w_{y_2},w_k) = \diam(T')$. 
    Applying Lemma \ref{lem:3.10} to $w_k$, $w_{y_2}$, $w_1$ and $w_{x_2}$ implies that $d(w_{x_2},w_1) = \diam(T')$. 
    Note that $w_1, w_{x_2}$ and $w_{y_2}$ are distinct since $T'$ is nontrivial with even diameter.
    Let $P$ be the $w_1-w_{x_2}$ path, $P'$ be the $w_{y_2}-w_{x_2}$ path and let $w_\ell\in V(P)\cap V(P')$ that is closest to $w_{y_2}$.  
    Then
    \[\begin{aligned}
        \diam(T') & = d_{T'}(w_1,w_{x_2}) = d_{T'}(w_1,w_\ell) + d_{T'}(w_\ell,w_{x_2}),  \\
        \diam(T') & = d_{T'}(w_{x_2},w_{y_2}) = d_{T'}(w_{x_2},w_\ell) + d_{T'}(w_\ell,w_{y_2}), \text{ and} \\
        \diam(T')-1 & = d_{T'}(w_1,w_{y_2}) = d_{T'}(w_1,w_\ell) + d_{T'}(w_\ell,w_{y_2}).
    \end{aligned}\]
    The first two equations imply that $d_{T'}(w_1,w_{\ell})=d_{T'}(w_{\ell},w_{y_2})$ and third imply that both equal $(\diam(T')-1)/2$, contradicting that $T'$ has even diameter.

    The case where Equations \ref{eq6} and \ref{eq7} hold is similar to when Equations \ref{eq5} and \ref{eq8} hold.  The case where Equations \ref{eq6} and \ref{eq8} hold is similar to when Equations \ref{eq5} and \ref{eq7} hold.  

    All cases end in a contradiction so $c$ is rainbow-free and $\aw(T\square T',3) = 4$. \end{proof}

The remaining results in this section consider conditions on tree products that give $\aw(T\square T',3) = 3$.  This is done by taking an arbitrary $3$-coloring and guaranteeing a rainbow $3$-AP. We begin with the smallest case when one of the trees is $P_2$, as it behaves quite differently from other trees.

    \begin{lem}\label{lem:annoyingP2}
        If $T$ is a tree with odd diameter, then $\aw(P_2\square T,3) = 3$.
        
    \end{lem}

    \begin{proof}
        Let $c$ be an exact, rainbow-free $3$-coloring of $P_2\square T$.
        By Lemma \ref{isometricpathorC3} we can find either an isometric path or a $C_3$ that contains all three colors.  
        No such $C_3$ exists since $P_2\square T'$ is bipartite, so let $P$ be a shortest isometric path that contains all three colors.
        Further, denote the first vertex of the path by $v_{a,b}$ and the last vertex of the path by $v_{c,d}$ with $a\le c$ and $b\le d$.  
        Without loss of generality, assume that $c(v_{a,b}) = red$, $c(v_{c,d}) = blue$ and note that every other vertex on the path is green.  
        Note that $a\neq c$ else we contradict Lemma \ref{|c(Hi)|<3} or the minimality of $P$. 
        Also, $P$ has odd length of at least three, else we can find a rainbow $3$-AP. 
        Since $P_2$ only has two vertices which are $u_a$ and $u_c$, for the remainder of the proof we will use $u_1=u_a$ and $u_2= u_c$.
        Note that the path $v_{1,b} - v_{2,b} - v_{2,d}$ and the path $v_{1,b} - v_{1,d}-v_{2,d}$ are both of shortest length and both have interior vertices that are all green, thus $green$ is in both $T_1$ and $T_2$.

        Since $d_T(u_1,u_2) = 1$, it follows that $d_{T'}(w_b,w_d) = d(v_{1,b},v_{2,d})-1$ which is even. 
        Additionally, since $\diam(T')$ is odd, we have $d_{T'}(w_b,w_d) < \diam(T')$.
        Without loss of generality, assume that $w_b$ is peripheral or $w_d$ is not peripheral.
        If $w_b$ is peripheral, then there exists some $w_{\ell} \in V(T)$ such that $d_{T}(w_b,w_{\ell}) = d_{T}(w_b,w_d)+1$. 
        If $w_d$ is not peripheral, then by Corollary \ref{d(u,v)+1}, there exists some $w_{\ell} \in V(T)$ such that $d_{T}(w_b,w_{\ell}) = d_{T}(w_b,w_d)+1$. 
        Consider $v_{2,\ell-1}$ where $w_{\ell-1}$ is the unique neighbor of $w_\ell$ on the $w_b-w_\ell$ path. 
        If $w_d = w_{\ell-1}$, then we have $c(v_{2,\ell-1}) = blue$. 
        If not, then the $3$-AP $(v_{2,d},v_{1,b},v_{2,\ell-1})$ implies $c(v_{c,\ell-1}) \neq green$ and Corollary \ref{|c(Hi)|<3} implies that $c(v_{2,\ell-1}) \neq red$ since $blue,green \in c(V(T_2))$. So, $c(v_{2,\ell-1}) = blue$.
        Now, the $v_{1,b}-v_{2,\ell}$ path contains all three colors and is a subgraph of some $P_2 \square P_n$ where $n$ is even because $d(w_b,w_\ell) = d(w_b,w_d)+1$ is odd. 
        This $P_2\square P_n$ subgraph is isometric, so Theorem \ref{PmxPn} implies $P_2 \square T$ has a rainbow $3$-AP, a contradiction.
    \end{proof}

    Recall that a tree $T$ with odd diameter is weakly non-$3$-peripheral if for every peripheral vertex $v$, $T_{v^-}$ is $3$-peripheral.

\begin{prop}\label{prop:weakly/odd}
    Suppose $T$ is a weakly non-$3$-peripheral tree with odd diameter. 
    Then $\aw(T\square T',3) = 3$ for any nontrivial tree $T'$ with odd diameter. 
\end{prop}

\begin{proof}
    Assume $c$ is a rainbow-free, exact $3$-coloring of $T \square T'$ using the colors $red$, $blue$, and $green$.
    By Lemma \ref{isometricpathorC3}, there is either an isometric $C_3$ or isometric path containing all three colors.
    Since $T\square T'$ is bipartite, it has no $C_3$ subgraph so there must be an isometric path containing all three colors. Let $P$ be a shortest such path. 
    Specifically, say $P$ is a $v_{a,b}-v_{c,d}$ path where $v_{a,b}$ is $red$, $v_{c,d}$ is $blue$ and all others are $green$ and, without loss of generality, $a\le c$ and $b \le d$. 
    Note that $P$ must have odd length, otherwise there is a rainbow $3$-AP.
    Note that $b \neq d$ and $a\neq c$ otherwise $T_b=T_d$ and $T'_a = T'_c$, respectively, contain all three colors contradicting Lemma \ref{|c(Hi)|<3}. 
    Define $u_{c-1} \in V(T)$ as the unique neighbor of $u_c$ on the $u_a-u_c$ path.
    Note that the path following $v_{a,b}-v_{a,d}-v_{c,d}$ is a $v_{a,b}-v_{c,d}$ geodesic. 
    So, $c(v_{i,d})=green$ for $a\le i \le c-1$ otherwise we can construct a path containing all three colors which is shorter than $P$.
    In particular, this shows $green$ is in every copy of $T'$ which intersects $P$ except for possibly $T'_c$.
    
    First consider the case where $d_T(u_a,u_c) < \diam(T)$.  
    Then there exists some $u_i \in V(T)$ such that $P$ is contained in  $T_{u_i^-} \square {T'}$. 
    In particular, if either $u_a$ or $u_c$ is peripheral, then choose that to be $u_i$, otherwise, if neither are peripheral, any peripheral $u_i$ suffices.
    Notice that  $T_{u_i^-} \square T'$ is an isometric subgraph of $T\square T'$ and contains three colors. 
    Since $T$ is weakly non-$3$-peripheral, $T_{u_i^-}$ is $3$-peripheral.  By Theorem \ref{TxG 3-per}, there is a rainbow $3$-AP in $T_{u_i^-} \square T'$.
    Since $T_{u_i^-} \square T'$ is an isometric subgraph, this contradicts that $c$ is a rainbow-free coloring. 
    
    Now consider when $d_T(u_a,u_c) = \diam(T)$.
    Since the length of $P$ is odd, it must be less than $\diam(T\square T')$. So, $d_{T'}(w_b,w_d) < \diam(T')$.
    Without loss of generality, assume that $w_b$ is peripheral or $w_d$ is not peripheral.
    If $w_b$ is peripheral, then there exists some $w_{\ell} \in V(T')$ such that $d_{T'}(w_b,w_{\ell}) = d_{T'}(w_b,w_d)+1$. 
    If $w_d$ is not peripheral, then by Corollary \ref{d(u,v)+1}, there exists some $w_{\ell} \in V(T')$ such that $d_{T'}(w_b,w_{\ell}) = d_{T'}(w_b,w_d)+1$. 
    Define $P'$ to be a $v_{a,b}-v_{c-1,\ell}$ geodesic in $T\square T'$ and note that $P$ and $P'$ have the same length. 
    We will show $c(v_{c-1,\ell})=blue$ and all interior vertices of $P'$ are colored $green$ by $c$.
    If $c(v_{c-1,\ell})=green$, then $(v_{c,d},v_{a,b},v_{c-1,\ell})$ is a rainbow $3$-AP, a contradiction. 
    If $c(v_{c-1,\ell})=red$, then since $green \in c(V(T'_{c-1}))$, we have $green,red,blue \in c(V(T'_{c-1}) \cup V(T'_c))$, contradicting Corollary \ref{|c(V(Gi U Gj))|<3}.
    Thus, $c(v_{c-1,\ell})=blue$. 
    Recall that $green$ appears in every copy of $T'_i$ for each $u_i$ on the $u_a-u_{c-1}$ path in $T$.
    Since we also have $red \in c(V(T'_a))$, $blue \in c(V(T'_{c-1}))$, repeated applications of Corollary \ref{|c(V(Gi U Gj))|<3} show that there is some $u_i \in V(T)$ on the $u_a-u_{c-1}$ path such that $c(V(T_i))=\{green\}$. 
    In particular, there is some interior vertex of $P'$ colored $green$. 
    So, if any other interior vertex of $P'$ were not colored $green$, then there exists a path containing all three colors which is shorter than $P$, contradicting the minimality of $P$. 
    Thus, all interior vertices of $P'$ are $green$. 
    Since $d_T(u_a,u_{c-1})<\diam(T)$ applying the argument in the previous paragraph to $P'$ yields a rainbow $3$-AP.
\end{proof}

Recall that a tree $T$ with even diameter is weakly non-$3$-peripheral if there exists some vertex $v\in V(T)$ such that $T_{v^+}$ is $3$-peripheral.

\begin{prop}\label{prop:weakly/even}
    Suppose $T$ is a weakly non-$3$-peripheral tree with even diameter. Then $\aw(T\square T',3) = 3$ for any nontrivial even diameter tree $T'$. 
\end{prop}

\begin{proof}
    Assume $c$ is a rainbow-free, exact $3$-coloring of $T \square T'$ using the colors $red$, $blue$, and $green$.
    By Lemma \ref{isometricpathorC3}, there is either a isometric $C_3$ or isometric path containing all three colors.
    Since $T\square T'$ is bipartite no such $C_3$ exists so let $P$ be a shortest such path. 
    Specifically, say $P$ is a $v_{a,b}-v_{c,d}$ path where $v_{a,b}$ is $red$, $v_{c,d}$ is $blue$, and all others are $green$.
    We remark that all $v_{a,b}-v_{c,d}$ paths must have the interior vertices be $green$, else we can construct a shorter path containing all three colors. 
    Furthermore, $P$ must have odd length, otherwise there is a rainbow $3$-AP.

    First consider the case when $d(u_a,u_c)=\diam(T)$.
    Then $d_T(u_a,u_c)$ is even implying that $d_{T'}(w_b,w_d)$ is odd, and in particular, $d_{T'}(w_b,w_d) < \diam(T')$.
    Without loss of generality, suppose $w_b$ is peripheral or $w_d$ is not peripheral.
        If $w_b$ is  peripheral, there exists a vertex $w_{d'}$ in $T'$ for which $d_{T'}(w_b,w_{d'}) = d_{T'}(w_b,w_d)+1$.
        If $w_d$ is not peripheral, Corollary \ref{d(u,v)+1} yields a vertex $w_{d'}$ in $T'$ for which $d_{T'}(w_b,w_{d'}) = d_{T'}(w_b,w_d)+1$.
        Since $T$ is weakly non-$3$-peripheral, there exists some vertex $u_{\ell}$ in $T$ such that $T_{u_\ell^+}$ is $3$-peripheral.   
        Lemma \ref{lem:weakly/per} implies $d_T(u_a,u_\ell)=d_T(u_c,u_\ell)=\diam(T)-1$. 
        This implies $(v_{c,d},v_{a,b},v_{\ell,d'})$ is a $3$-AP. 
        To avoid a rainbow $3$-AP, $v_{\ell,d'}$ is not $green$. Since $v_{a,d}$ and $v_{c,b}$ lie on a $v_{a,b}-v_{c,d}$ geodesic, an earlier remark implies they must be $green$.
        Now, the $3$-APs $(v_{a,d},v_{\ell,d'},v_{c,d})$ or $(v_{a,b},v_{\ell,d'},v_{c,d})$ are rainbow depending on whether $v_{\ell,d'}$ is $red$ or $blue$, respectively, contradicting that $c$ is a rainbow free coloring.
        
        Now consider the case when $d(u_a,u_c)<\diam(T)$.
        Applying the same argument used in the previous case, there exists some $u_{c'}$ in $T$ such that, without loss of generality, $d_T(u_a,u_{c'}) = d(u_a,u_c)+1$.
         Let $w_{d-1} \in V(T')$ be the unique neighbor of $w_d$ on the $w_b-w_d$ path.
         Consider the vertex $v_{c',d-1}$ and the $3$-AP $(v_{c,d},v_{a,b},v_{c',d-1})$. 
         Notice $v_{c',d-1}$ cannot be $green$, otherwise this $3$-AP is rainbow.
         Since $T'$ is nontrivial and has even diameter, $|V(T')| \geq 3$. 
         So, Corollary \ref{|c(V(Gi U Gj))|<3} implies that $v_{c',d-1}$ cannot be $red$, otherwise $|c(V(T_{d-1}) \cup V(T_d))| = 3$ because $green$ must appear in $T_{d-1}$ or $T_{d}$.
         Thus, $c(v_{c',d-1}) = blue$.
        
        Let $P'$ be a $v_{a,b}-v_{c',d-1}$ geodesic.
        Since the length of $P'$ is the same as the length of $P$, all internal vertices of $P$ must be colored $green$, otherwise there is a shorter path than $P$ containing all three colors.
        If $d_T(u_a,u_{c'}) = \diam(T)$, then apply the argument from the previous case.
        So, suppose $d_T(u_a,u_{c'}) < \diam(T)$.
        Also, $d(w_b,w_{d-1}) < d(w_b,w_d) \leq \diam(T')$.
        So, there are some peripheral vertices $u_r \in V(T), w_i \in V(T')$ such that $P'$ is contained in $T_{u_r^-} \square T'_{w_i^-}$.
        We choose $u_r$ to be equal $u_a$ or $u_{c'}$ if either are peripheral, otherwise we may choose $u_r$ to be any peripheral vertex. 
        Likewise for choosing $w_i$ depending on whether $w_b$ and $w_{d-1}$ are peripheral or not.
        Recall $T_{u_{\ell}^+}$ is $3$-peripheral. 
        Since $u_r$ is a peripheral vertex of $T$, there exists some $u_s \in V(T)$ such that $\dist_T(u_r,u_s)=\diam(T)$.
        Lemma \ref{lem:weakly/per} implies that $d(u_\ell,u_r) = d(u_\ell,u_s) = \diam(T)-1$.
        Let $u_{r-1}$ and $u_{s-1}$ be the unique neighbors of $u_r$ and $u_s$, respectively, so that \[d(u_{r-1},u_{s-1}) = d(u_{\ell},u_{r-1}) = d(u_\ell,u_{s-1}) = \diam(T)-2 = \diam\left(\left(T_{u_r^-}\right)_{u_{\ell}^-}\right),\]
        where the last equality holds by Observation \ref{obs:diamTtilde}.
        If $\diam(T)=2$, then $u_{r-1}=u_{s-1}=u_{\ell}$, and if $\diam(T)>2$, then these three vertices are pairwise distinct.
        So, $\left(T_{u_r^-}\right)_{u_{\ell}^-}$ is a single vertex or $3$-peripheral.
        This implies $T_{u_r^-} \cong P_2$ or $T_{u_r^-}$ is weakly non-$3$-peripheral.
        By Observation \ref{obs:diamTtilde}, Lemma \ref{lem:annoyingP2} and Proposition \ref{prop:weakly/odd}, we have $\aw\left(T_{u_r^-} \square T'_{w_i^-},3\right)=3$.
        Since $P'$ contains all three colors, it follows that $T_{u_r^-} \square T'_{w_i^-}$ contains all three colors and thus contains a rainbow $3$-AP.
        Since this is an isometric subgraph of $T\square T'$, this rainbow $3$-AP also exists in $T\square T'$, a contradiction.
\end{proof}

Now, using Theorems \ref{thm:treeprodeven}, \ref{TxG 3-per} and \ref{theorem:treeprododd}, we have a full classification of the anti-van der Waerden number of the product of two trees.  This quickly leads to a full classification of products of two forests.

\begin{customcor}{1.4}

Let $T$ and $T'$ be trees.  Then,
\[\aw(T\square T',3) = 
\begin{cases}
     & \text{if $T$ or $T'$ is $3$-peripheral, or}\\

    3& \text{$\diam(T\square T')$ is even and $T$ or $T'$ is $P_2$, or}\\

    & \text{$\diam(T\square T')$ is even and $T$ or $T'$ is weakly non-$3$-peripheral,}\\
    4 & \text{otherwise}.
\end{cases}\]

\end{customcor}

The following observation relies on the pigeonhole principle and helps to provide a full classification for products of forests.
\begin{obs}[\cite{SWY}]\label{obs:disconnected}
    If $G$ is disconnected with connected components $\{G_i\}_{i=1}^\ell$, then
    \[\aw(G,k) = 1+\displaystyle\sum_{i=1}^\ell (\aw(G_i,k) - 1).\]
\end{obs}

Corollary \ref{cor:full} and Observation \ref{obs:disconnected} gives the following Corollary. Note that Corollary \ref{cor:forest} was previously known, but it was not known how to compute $|P|$ and $|S|$ explicitly.

\begin{cor}\label{cor:forest}
    Let $F_1$ and $F_2$ be forests and let $P$ be the set of connected components of $F_1\square F_2$ whose anti-van der Waerden number is $3$ and $S$ be the set of connected components of $F_1\square F_2$ whose anti-van der Waerden number is $4$.  Then,

    \[\aw(F_1\square F_2,3) = 2|P| + 3|S| + 1.\]
\end{cor}

\subsection*{Acknowledgements} The second and third authors were supported by the Dean's Distinguished Fellowship research grant from the University of Wisconsin-La Crosse.  The first author was supported by NSF grant DMS-2418903 and summer research stipend from Fairfield University.

\end{document}